\documentclass[a4paper]{amsart}
\usepackage[utf8x]{inputenc}
\usepackage{graphicx}
\usepackage{enumerate}
\newtheorem{MainTheorem}{Theorem}

\newtheorem{theorem}{Theorem}[section]

\newtheorem{lemma}[theorem]{Lemma}
\newtheorem{corollary}[theorem]{Corollary}
\newtheorem*{ST}{Sharkovski\u{\i} Theorem for maps from $\cSO$}

\theoremstyle{definition}
\newtheorem{definition}[theorem]{Definition}
\newtheorem{remark}[theorem]{Remark}
\newcommand{\openthickbox}{\leavevmode\hbox to.77778em{%
  \hfil\vrule width.175em
  \vbox to.675em{\hrule width.26em height.175em\vfil\hrule height.175em}%
  \vrule width.175em\hfil}}
\newcommand{\qedple}{\leavevmode\unskip\penalty9999\hbox{}\nobreak\hfill\quad\hbox{\openthickbox}}
\makeatletter
\newcommand{\TeoremaAmbFinalMarcat}[1]{%
  \expandafter\gdef\csname end#1\endcsname{\qedple\@endtheorem}}
\TeoremaAmbFinalMarcat{definition}
\TeoremaAmbFinalMarcat{remark}
\makeatother
\newcommand{\evalat}[1]{\bigr\rvert_{#1}}
\newcommand{\set}[2]{\ensuremath{\{#1 \,\colon #2\}}}
\makeatletter
\def\@map#1#2[#3]{\mbox{$#1 \colon #2 \longrightarrow #3$}}
\def\map#1#2{\@ifnextchar [{\@map{#1}{#2}}{\@map{#1}{#2}[#2]}}
\def\@chull#1[#2]{\ensuremath{\langle #1 \rangle_{_{#2}}}}
\def\chull#1{\@ifnextchar [{\@chull{#1}}{\ensuremath{\langle #1 \rangle}}}
\def\@ls#1#2{\@mathmeasure\z@\displaystyle{#2}
    \@mathmeasure\@ne\displaystyle{#1}\box\@ne\box\z@}
\def\gtso#1{\mathrel{\@ls{_{#1}}{ > }}}
\def\geso#1{\mathrel{\@ls{_{#1}}{ \ge }}}
\def\ltso#1{\mathrel{ < _{#1}}}

\makeatother
\def\Sho{\mbox{\tiny\textup{Sh}}}
\newcommand{\N}{\ensuremath{\mathbb{N}}}
\newcommand{\Z}{\ensuremath{\mathbb{Z}}}
\newcommand{\Q}{\ensuremath{\mathbb{Q}}}
\newcommand{\R}{\ensuremath{\mathbb{R}}}
\newcommand{\I}{\ensuremath{\mathbb{I}}}
\newcommand{\SI}{\ensuremath{\mathbb{S}^1}}
\newcommand{\cB}{\mathcal{B}}
\DeclareMathOperator{\Sign}{Sign}
\DeclareMathOperator{\Card}{Card}
\DeclareMathOperator{\diam}{diam}
\DeclareMathOperator{\Int}{Int}
\DeclareMathOperator{\Bd}{Bd}

\DeclareMathOperator{\Graph}{Graph}
\DeclareMathOperator{\tr}{tr}
\providecommand{\abs}[1]{\lvert#1\rvert}
\newcommand{\signedarrow}[2][\empty]{\nolinebreak[4]\xrightarrow[#1]{\:#2\:}\nolinebreak[4]}
\newcommand{\arrowequal}{\mathrel{=}\mathrel{\mkern-3.5mu}\mathrel{=}\mathrel{\mkern-2.5mu}\relbar\mathrel{\mkern-3.5mu}\rightarrow}
\newcommand{\signedarrowequal}[2][\empty]{\nolinebreak[4]\ifx#1\empty\overset{#2}{\arrowequal}\else\underset{#1}{\overset{#2}{\arrowequal}}\fi\nolinebreak[4]}
\newcommand{\cSO}{\ensuremath{\mathcal{S}(\Omega)}}
\newcommand{\cI}{\ensuremath{\mathcal{C}^{0}(\I,\I)}}
\newcommand{\rescont}[1][M]{G_{_{#1}}}
\title{Forcing and entropy of strip patterns \\ of quasiperiodic skew products in the cylinder}
\author[Ll. Alsed\`a]{Llu\'{\i}s Alsed\`a}
\address{Departament de Matem\`atiques,
Edifici Cc,
Universitat Aut\`onoma de Barcelona,
08913 Cerdanyola del Vall\`es,
Barcelona,
Spain}
\email{alseda@mat.uab.cat}

\author[F. Ma\~{n}osas]{Francesc Ma\~{n}osas}
\address{Departament de Matem\`atiques,
Edifici Cc,
Universitat Aut\`onoma de Barcelona,
08913 Cerdanyola del Vall\`es,
Barcelona,
Spain}
\email{manyosas@mat.uab.cat}

\author[L. Morales]{Leopoldo Morales}
\address{Departament de Matem\`atiques,
Edifici Cc,
Universitat Aut\`onoma de Barcelona,
08913 Cerdanyola del Vall\`es,
Barcelona,
Spain}
\email{mleo@mat.uab.cat}
\thanks{The authors have been partially supported by MINECO grant
numbers MTM2008-01486 and MTM2011-26995-C02-01}
\date{November 11, 2014}
\begin{document}
\begin{abstract}
We extend the results and techniques from \cite{FJJK} to study the
combinatorial dynamics (\emph{forcing}) and entropy of
quasiperiodically forced skew-products on the cylinder. For these maps
we prove that a cyclic permutation $\tau$ forces a cyclic permutation
$\nu$ as interval patterns if and only if $\tau$ forces $\nu$ as
cylinder patterns.
This result gives as a corollary the Sharkovski\u{\i} Theorem for
quasiperiodically forced skew-products on the cylinder proved in \cite{FJJK}.

Next, the notion of $s$-horseshoe is defined for quasiperiodically
forced skew-products on the cylinder and it is proved, as in the
interval case, that if a quasiperiodically forced skew-product on the
cylinder has an $s$-horseshoe then its topological entropy is larger
than or equals to $\log(s).$

Finally, if a quasiperiodically forced skew-product on the cylinder
has a periodic orbit with pattern $\tau,$ then $h(F) \ge h(f_{\tau}),$
where $f_{\tau}$ denotes the \emph{connect-the-dots} interval map over
a periodic orbit with pattern $\tau.$
This implies that if the period of $\tau$ is $2^n q$ with $n \ge 0$
and $q \ge 1$ odd, then $h(F) \ge \tfrac{\log(\lambda_q)}{2^n}$,
where $\lambda_1 = 1$ and,
for each $q \ge 3,$ $\lambda_q$ is the largest root of the
polynomial $x^q − 2x^{q−2} − 1.$
Moreover,
for every $m=2^n q$ with $n \ge 0$ and $q \ge 1$ odd,
there exists a quasiperiodically forced skew-product on the cylinder $F_m$
with a periodic orbit of period $m$ such that $h(F_m) = \tfrac{\log(\lambda_q)}{2^n}.$
This extends the analogous result for interval maps to
quasiperiodically forced skew-products on the cylinder.
\end{abstract}
\maketitle
\section{Introduction}
In this paper we want to study the  coexistence and implications
between periodic objects of maps on the cylinder
$\Omega = \SI\times \I,$ of the form:
\[
\map{F}{\begin{pmatrix} \theta \\ x\end{pmatrix}}[
   {\begin{pmatrix} R_\omega(\theta)\\f(\theta,x)\end{pmatrix}}
],
\]
where $\SI = \R / \Z$, $\I = [0,1]$,
$R_\omega(\theta) = \theta + \omega \pmod{1}$ with
$\omega \in \R \setminus \Q$
and $f(\theta,x) = f_{\theta}(x)$ is continuous on both variables.
To study this class of maps, in \cite{FJJK}, were developed clever
techniques that lead to a theorem of the Sharkovski\u{\i} type for
this class of maps and periodic orbits of appropriate objects.

We aim at extending these results and techniques to study the
combinatorial dynamics (\emph{forcing}) and entropy of the
skew-products from the class {\cSO} consisting on all maps of the
above type.

As already remarked in \cite{FJJK}, instead of $\SI$ we could take any
compact metric space~$\Theta$ that admits a minimal homeomorphism
$\map{R}{\Theta}$ such that
$R^{\ell}$ is minimal for every $\ell > 1$.
However, for simplicity and clarity we will remain in the class
$\cSO.$

Before stating the main results of this paper, we will recall the
extension of Sharkovski\u{\i} Theorem to {\cSO} from \cite{FJJK},
together with the necessary notation.
We start by clarifying the notion of a periodic orbit for maps from
{\cSO}. To this end we informally introduce some key notions that
will be defined more precisely in Section~\ref{secDefRes}.

Let $X$ be a compact metric space.
A subset $G \subset X$ is \emph{residual} if it
contains the intersection of a countable family of open dense subsets
in $X.$

In what follows, $\map{\pi}{\Omega}[\SI]$ will denote the
standard projection from $\Omega$ to the circle.

Instead of periodic points we use objects that project over the
whole~$\SI,$ called \emph{strips} in \cite[Definition~3.9]{FJJK}.
A \emph{strip in $\Omega$} is a closed set $B \subset \Omega$ such
that $\pi(B) = \SI$ (i.e., $B$ projects on the whole $\SI$) and
$\pi^{-1}(\theta) \cap B$ is a closed interval (perhaps degenerate to a point)
for every $\theta$ in a residual set of $\SI.$

Given two strips $A$ and $B,$ we will write $A < B$ and $A \le B$
(\cite[Definition~3.13]{FJJK}) if there exists a residual set
$G \subset \SI,$ such that
for every $(\theta,x) \in A \cap \pi^{-1}(G)$
and $(\theta,y) \in B \cap \pi^{-1}(G)$
it follows that $x < y$ and, respectively, $x \le y$.
We say that the strips $A$ and $B$ are
\emph{ordered}\footnote{This notion will be defined with greater detail
but equivalently in Definition~\ref{orden}.
We are giving here this less technical definition just to simplify this general section.}
(respectively \emph{weakly ordered})
if either $A < B$ or $A > B$
(respectively $A \le B$ or $A \ge B$).

Given $F \in \cSO$ and $n \in \N$, a strip $B \subset \Omega$ is
called \emph{$n$-periodic} for $F$ (\cite[Definition~3.15]{FJJK}),
if $F^{n}(B) = B$ and the image sets
$B,\ F(B),\ F^{2}(B),\dots, F^{n-1}(B)$
are pairwise disjoint and pairwise ordered.

To state the main theorem of \cite{FJJK} we need to recall the
\emph{Sharkovski\u{\i} Ordering} (\cite{Shar, Shartrans}).
The \emph{Sharkovski\u{\i} Ordering} is a linear ordering of
$\N$ defined as follows:
\begin{align*}
& 3 \gtso{\Sho} 5 \gtso{\Sho} 7 \gtso{\Sho} 9 \gtso{\Sho} \dots \gtso{\Sho} \\
& 2 \cdot 3 \gtso{\Sho} 2 \cdot 5 \gtso{\Sho} 2 \cdot 7 \gtso{\Sho} 2 \cdot 9 \gtso{\Sho} \dots \gtso{\Sho} \\
& 4 \cdot 3 \gtso{\Sho} 4 \cdot 5 \gtso{\Sho} 4 \cdot 7 \gtso{\Sho} 4 \cdot 9 \gtso{\Sho} \dots \gtso{\Sho}\\
& \hspace*{7em} \vdots \\
& 2^n \cdot 3  \gtso{\Sho}  2^n \cdot 5 \gtso{\Sho} 2^n \cdot 7 \gtso{\Sho} 2^n \cdot 9 \gtso{\Sho} \dots \gtso{\Sho} \\
& \hspace*{7em} \vdots \\
& \cdots \gtso{\Sho} 2^n \gtso{\Sho} \dots \gtso{\Sho}
  16 \gtso{\Sho} 8 \gtso{\Sho} 4 \gtso{\Sho} 2 \gtso{\Sho} 1.
\end{align*}
In the ordering $\geso{\Sho}$ the least element is 1 and the largest
one is 3. The supremum of the set $\{1,2,4,\dots,2^n,\dots\}$ does
not exist.

\begin{ST}[\cite{FJJK}]
Assume that the map $F \in \cSO$ has a $p$-periodic strip.
Then $F$ has a $q$-periodic strip for every $q \ltso{\Sho} p.$
\end{ST}

Our first main result (Theorem~\ref{teo-prin-A}) concerns the forcing
relation.
As we will see in detail, the \emph{strips patterns} of
periodic orbits of strips of maps from $\cSO$ can be formalized in a natural
way as cyclic permutations, as in the case of the periodic patterns
for interval maps.
Our first main result states that a cyclic permutation
$\tau$ forces a cyclic permutation $\nu$ as interval patterns if and
only if $\tau$ forces $\nu$ as strips patterns.

Since the Sharkovski\u{\i} Theorem in the interval follows from the
forcing relation, a corollary of Theorem~\ref{teo-prin-A} is
the Sharkovski\u{\i} Theorem for maps from $\cSO$.

Next, an $s$-horseshoe for maps from $\cSO$ can be defined also in a
natural way. Our second main result (Theorem~\ref{teo-prin-B}) states
that if a map $F \in \cSO$ has an $s$-horseshoe then $h(F)$, the
\emph{topological entropy of $F$}, satisfies $h(F) \ge \log(s).$
This is a generalization of the well known result for the interval.

The third main result of the paper (Theorem~\ref{teo-prin-C}) states
that if a map $F \in \cSO$ has a periodic orbit of strips with
strips pattern $\tau,$ then $h(F) \ge h(f_{\tau}),$
where $f_{\tau}$ denotes the \emph{connect-the-dots} interval map
over a periodic orbit with pattern $\tau$.
A corollary of this fact and the lower bounds of the
topological entropy of interval maps from \cite{BGMY} is that,
if the period of $\tau$ is $2^n q$ with $n \ge 0$ and $q \ge 1$ odd,
then $h(F) \ge \tfrac{\log(\lambda_q)}{2^n}$, where $\lambda_1 = 1$
and, for each $q \ge 3,$ $\lambda_q$ is the largest root of the
polynomial $x^q − 2x^{q−2} − 1.$
Moreover,
for every $m=2^n q$ with $n \ge 0$ and $q \ge 1$ odd,
there exists a quasiperiodically forced skew-product on the cylinder $F_m$
with a periodic orbit of strips of period $m$ such that
$h(F_m) = \tfrac{\log(\lambda_q)}{2^n}.$

The paper is organized as follows. In Section~\ref{secDefRes}
we introduce the notation and we state the results in detail and in
Section~\ref{ProofOfteo-prin-A} we prove Theorem~\ref{teo-prin-A}
Finally, in Section~\ref{ProofOfCandD} we prove
Theorems~\ref{teo-prin-B} and \ref{teo-prin-C}.

\section{Definitions and statements of results}\label{secDefRes}

We start by recalling the notion of interval pattern and related
results. Afterwards we will introduce the natural extension to the
class $\cSO$ by defining the cylinder patterns.

In what follows we will denote the class of continuous maps from the
interval $\I$ to itself by $\cI.$

\subsection{Interval patterns}\label{IntPat}

Given $f \in \cI,$ we say that $p \in \I$ is an $n$-periodic point of
$f$ if $f^{n}(p) = p$
and $f^{j}(p) \ne p$ for $j = 1,2,\dots,n-1.$
The set of points $\{p,f(p),f^{2}(p),\ldots,f^{n-1}(p)\}$ will be
called a \emph{periodic orbit}.
A periodic orbit $P = \{p_1,p_2,\ldots,p_n\}$ is said to have the
\emph{spatial labelling} if $p_{1} < p_{2} < \dots < p_{n}.$
In what follows, every periodic orbit will be assumed to have
the spatial labelling unless otherwise stated.

\begin{definition}[Interval pattern]
Let $P = \{p_{1} < p_{2} < \dotsb < p_{n}\}$ be a periodic orbit of a
map $ f \in \cI$ and let $\tau$ be a cyclic permutation over
$\{1,2,\ldots,n\}.$
The periodic orbit $P$ is said to have the
\emph{(periodic) interval pattern} $\tau$ if and only if
$f(p_{i}) = p_{\tau(i)}$ for $i = 1,2,\dots,n.$
The period of $P$, $n,$ will also be called the
\emph{period of $\tau$}.
\end{definition}

\begin{remark}
Every cyclic permutation can occur as interval pattern.
\end{remark}

To study the dynamics of functions from $\cI$ we introduce the
following ordering on the set of interval patterns.

\begin{definition}[Forcing]
Given two interval patterns $\tau$ and $\nu,$  we say that
\emph{$\tau$ forces $\nu,$ as interval patterns}, denoted by
$\tau \Longrightarrow_{\I} \nu,$
if and only if \emph{every} $f \in \cI$ that has a periodic
orbit with interval pattern~$\tau$ also has a periodic orbit
with interval pattern~$\nu.$ By \cite[Theorem~2.5]{ALM},
the relation $\Longrightarrow_{\I}$ is a partial ordering.
\end{definition}

Next we define a \emph{canonical map} for an interval pattern as
follows.

\begin{definition}[$\tau$-linear map]\label{connect-the-dots}
Let $f \in \cI$ and let $P = \{p_{_1},p_{_2},\dots,p_{_n}\}$ be an
$n$-periodic orbit of $f$ with the spatial labelling
($p_{_1} < p_{_2} < \dots < p_{_n}$).
We define the \emph{$P$-linear map $f_P$} as the unique map in $\cI$
such that
$f_P(p_{_i}) = f(p_{_i})$ for $i = 1,2,\ldots,n,$
$f_P$ is affine in each interval of the form $[p_i,p_{i+1}]$ for $i =
1,2,\ldots,n-1,$ and
$f_P$ is constant on each of the two connected components of $\I
\setminus [p_1,p_n]$.
The map $f_P$ is also called \emph{$P$-connect-the-dots map}.

Observe that the map $f_P$ is uniquely determined by $P$ and
$f\evalat{P}$.

Let $\tau$ be the pattern of the periodic orbit $P.$
The map $f_P$ will also be called a \emph{$\tau$-linear map} and
denoted by $f_{\tau}.$ Then the maps $f_\tau$ are not unique
but all maps $f_{\tau}\evalat{[\min P, \max P]}$ are topologically
conjugate and, thus, they have the same topological entropy and
periodic orbits.
\end{definition}

The next result is a useful characterization of the forcing relation
of interval patterns in terms of the $\tau$-linear maps.

\begin{theorem}[Characterization of the forcing
relation]\label{carat-forc}
Let $\tau$ and $\nu$ be two interval patterns.
Then, $\tau \Longrightarrow_{\I} \nu$ if and only if
$f_{\tau}$ has a periodic orbit with pattern~$\nu$.
\end{theorem}

\subsection{Strips Theory}
In this subsection we introduce a new (more restrictive) kind of strips
with better properties and we study the basic properties that we will need
throughout the paper.
To introduce this new kind of strips we first need to introduce the notion of a
\emph{core} of a set.

Given a compact metric space $(X,d)$ we denote the set of all closed (compact)
subsets of $X$ by $2^X,$ and we endow this space with the Hausdorff metric
\begin{align*}
H_{d}(B,C)
  &= \max\{\max_{b\in B}\min_{c\in C} d(c,b),
     \max_{c\in C}\min_{b\in B}d (c,b)\}\\
  &= \max\{\max_{b\in B} d(b,C), \max_{c\in C} d(c,B)\}.
\end{align*}
It is well known that $(2^X, H_d)$ is compact.
% In particular we have defined the spaces $(2^{\Omega}, H_d),$
% where $d$ denotes the standard metric in $\Omega,$ and
% $(2^{\I}, H_{|\cdot|}).$
Also, given a set $A$ we will denote the closure of $A$ by $\overline{A}$.

\begin{definition}[\cite{FJJK}]\label{core}
Let $M$ be a subset of $2^{\Omega}.$ We define the \emph{core of $M$}, denoted $M^c$, as
\[
   \bigcap_{G\in\mathcal{G}(\SI)} \overline{M\cap\pi^{-1}(G)},
\]
where $\mathcal{G}(\SI)$ denotes the set of all residual subsets of $\SI.$
Observe that if $M$ is compact, then
$M^c \subset M$ and,
$\pi(M) = \SI$ implies $\pi(M^c) = \SI.$
\end{definition}

This definition of \emph{core} is rather intricate.
Below we settle an equivalent and more useful definition
in the spirit of Lemma~3.2 and Remark~3.3 of \cite{FJJK}.
The role of the residual of continuity in this equivalent definition
is stated without proof in \cite{FJJK} and, hence,
we include the proof for completeness.

Let $M \in 2^{\Omega}$ be such that $\pi(M) = \SI.$
We define the map {\map{\phi_{_M}}{\SI}[2^{\I}]} by
$\phi_{_M}(\theta) := M^{\theta},$ and
$\rescont := \set{\theta \in \SI}{\text{$\phi_{_M}$ is continuous at $\theta$}}.$
It can be easily seen that $\phi_{_M}$ is upper semicontinuous (i.e.
for every $\theta \in \SI$ and every open $U \subset \I$ such that
$\phi_{_M}(\theta) \subset U,$ $\set{z \in \SI}{\phi_{_M}(z) \subset U}$
is open in $\SI$). Hence, by \cite[Theorem~7.10]{Choq},
the set $\rescont$ is residual.
The set $\rescont$ will be called the \emph{residual of continuity of $M$}.

Given $G \subset \SI$ and a map $\map{\varphi}{G}[2^{\I}]$,
$\Graph(\varphi) := \set{(\theta,\varphi(\theta))}{\theta \in G} \subset \SI \times 2^{\I}$
denotes the \emph{graph of $\varphi$}.
By abuse of notation we will identify $\Graph(\varphi)$
with the set $\bigcup_{\theta\in G} \{\theta\} \times \varphi(\theta)$.
Hence, we will consider $\Graph(\varphi)$ as a subset of $\Omega$ (or of $G \times \I$),
and $\overline{\Graph(\varphi)}$ is a compact subset of $\Omega.$

\begin{lemma}\label{CoreForResidualOfContinuity}
Let $M$ be a compact subset of $\Omega.$ Then,
\[
M^c = \overline{\Graph\left(\phi_{_M}\evalat{G}\right)} = \overline{M \cap \pi^{-1}(G)}
\]
for every residual set $G \subset \rescont.$
Moreover, $M \cap \pi^{-1}(G) = M^c \cap \pi^{-1}(G)$ and $\left(M^c\right)^c = M^c.$
\end{lemma}

\begin{proof}
We start by proving the first statement of the lemma. Notice that if
\begin{equation}\label{CoreForResidualOfContinuity-eq1}
\overline{M \cap \pi^{−1}(\rescont)} \subset \overline{M \cap \pi^{−1}(H)}
\quad\text{for every $H\in \mathcal{G},$}
\end{equation}
then
\[
   \overline{M \cap \pi^{−1}(G)} \subset \overline{M \cap\pi^{−1}(\rescont)} \subset
   M^c = \bigcap_{H\in\mathcal{G}} \overline{M \cap \pi^{−1}(H)} \subset
   \overline{M \cap \pi^{−1} (G)}.
\]
Hence, we only have to prove \eqref{CoreForResidualOfContinuity-eq1}.

Let $H \in \mathcal{G}$ and let $(\theta, x) \in M \cap \pi^{−1}(\rescont)$
(i.e. $\theta \in \rescont$ and $(\theta,x) \in M^{\theta} = \phi_{_M}(\theta)$).
Since $H$ is residual, it is dense in $\SI.$
Therefore, there exists a sequence $\{\theta_n\}_{n=1}^{\infty} \subset H$ converging to $\theta.$
Since $\theta \in \rescont$, $\phi_{_M}$ is continuous in $\theta.$
So, $\lim \phi_{_M}(\theta_n) = \phi_{_M}(\theta)$
and, for every $\varepsilon > 0$ exists $N\in \N$ such that
$
  d\bigl((\theta,x), \phi_{_M}(\theta_n)\bigr) \le
  H_d(\phi_{_M}(\theta), \phi_{_M}(\theta_n)) < \varepsilon
$
for every $n \ge N.$
Since the sets $\phi_{_M}(\theta_n)$ are compact, for every $n\in \N,$
there exists
$
(\theta_n,x_n) \in \phi_{_M}(\theta_n) \subset M \cap \pi^{−1}(H)
$
such that
$
d\bigl((\theta,x), (\theta_n,x_n)\bigr) = d\bigl((\theta,x), \phi_{_M}(\theta_n)\bigr).
$
Thus, $\lim(\theta_n, x_n) = (\theta,x)$ and, hence,
$(\theta,x) \in \overline{M \cap \pi^{−1}(H)}.$
This implies
$
M \cap \pi^{−1}(\rescont) \subset \overline{M \cap \pi^{−1}(H)}
$
which, in turn, implies \eqref{CoreForResidualOfContinuity-eq1}.

By the first statement,
\[
 M \cap \pi^{-1}(G) \subset \overline{M \cap \pi^{-1}(G)} \cap \pi^{-1}(G) = M^c \cap \pi^{-1}(G) \subset M \cap \pi^{-1}(G).
\]
Now, to end the proof of the lemma, take $\widetilde{G} = \rescont \cap \rescont[M^c],$
which is a residual set. By the part of the lemma already proven we have,
\[
 \left(M^c\right)^c = \overline{M^c \cap \pi^{-1}\bigl(\widetilde{G}\bigr)} = \overline{M \cap \pi^{-1}\bigl(\widetilde{G}\bigr)} = M^c.
\]
\end{proof}

Now we are ready to define the notion of band.

\begin{definition}[Band]\label{definition-band}
Every strip $A \subset \Omega$ such that $A^c = A$ will be called a \emph{band}.
\end{definition}

\begin{remark}\label{pseudoband}
In view of Lemma~\ref{CoreForResidualOfContinuity}
a band could be equivalently defined as follows:
A \emph{band} is a set of the form $\overline{\Graph(\varphi)},$
where  $\varphi$ is a continuous map from a residual set of $\SI$
to the connected elements (intervals) of $2^{\I}.$
\end{remark}

Given $F \in \cSO,$ a strip $A$ is
\emph{$F-$invariant} if $F(A) \subset A$ and
\emph{$F-$strongly invariant} if $F(A) = A.$
As usual, the intersection of two $F-$invariant strips is
either empty or an $F-$invariant strip.
An invariant strip is called \emph{minimal} if it does not have a strictly contained invariant strip.

\begin{remark}\label{band-properties}
From Corollary~3.5 and Lemmas~3.10 and 3.11 of \cite{FJJK} it follows that the bands in $\Omega$
have the following properties for every map from $\cSO$:
\begin{enumerate}[(1)]
\item The image of a band is a band.
\item Every invariant strip contains an invariant minimal strip.
\item Every invariant minimal strip is a strongly invariant band.
\end{enumerate}
\end{remark}

Moreover, the Sharkovski\u{\i} Theorem for maps from $\cSO$ is indeed,

\begin{ST}[\cite{FJJK}]
Assume that $F \in \cSO$ has a $p$-periodic strip.
Then $F$ has a $q$-periodic band for every $q \ltso{\Sho} p.$
\end{ST}

Next we introduce a particular kind of bands that play a key role in this theory
since they allow us to better study and work with the bands.

Given a set $A \subset \Omega$ and $\theta \in \Omega$
we will denote the set $A \cap \pi^{-1}(\theta)$ by $A^{\theta}.$

\begin{definition}
A band $A$ is called \emph{solid} when $A^{\theta}$ is an interval for every $\theta \in \SI$ and
$\delta(A) := \inf\set{\diam(A^{\theta})}{\theta \in \SI} > 0.$
Also, $A$ is called \emph{pinched} if $A^{\theta}$ is a point for
each $\theta$ in a residual subset of $\SI.$
\end{definition}

\begin{remark}\label{solidorpseudocurve}
From \cite[Theorem~4.11]{FJJK} it follows that there are only two kind
of strongly invariant bands: \emph{solid} or \emph{pinched}.
\end{remark}

Despite of the fact that the above notion of pinched band is completely natural,
for several reasons that will become clear later (see also \cite{AMM})
we prefer to view the pinched bands as \emph{pseudo-curves} in the spirit of
Remark~\ref{pseudoband}:

\begin{definition}
Let $G$ be a residual set of $\SI$ and let $\map{\varphi}{G}[\I]$
be a continuous map from $G$ to $\I.$
The set $\overline{\Graph(\varphi)}$ will be called a \emph{pseudo-curve}.
\end{definition}

The next remark summarizes the basic properties of the pseudo-curves.

\begin{remark}\label{pseudocurves-properties}
The following properties of the pseudo-curves are easy to prove:
\begin{enumerate}[(1)]
\item Every pseudo-curve is a band.
In particular $\pi\left(\overline{\Graph(\varphi)}\right) = \SI.$
\item The image of a pseudo-curve is a pseudo-curve.
Moreover, every invariant pseudo-curve is strongly invariant and minimal.
\end{enumerate}
Now assume that $\overline{\Graph(\varphi)}$ is a pseudo-curve
where $\varphi$ is a map from $G$ to $\I.$ Then,
\begin{enumerate}[(1)]\setcounter{enumi}{2}
\item $\rescont[\overline{\Graph(\varphi)}] \supset G$ (see e.g. \cite[Lemma~7.2]{Nadler}).
\item $\overline{\Graph(\varphi)} \cap \pi^{-1}(G) = \Graph(\varphi).$
\end{enumerate}
\end{remark}

Next we want to define a partial ordering in the set of strips.
We recall that a map $g$ from $\SI$ to $\I$ is
\emph{lower semicontinuous} (respectively \emph{upper semicontinuous})
at $\theta \in \SI$ if
for every $\lambda < g(\theta)$ (respectively $ \lambda > g(\theta)$)
there exists a neighbourhood $V$ of $\theta$ such that
$\lambda < g(V)$ (respectively $\lambda > g(V)$).
When this condition holds at every point in $\SI$ $g$ is said to be
\emph{lower semicontinuous on $\SI$}
(respectively \emph{upper semicontinuous on $\SI$}).

\begin{definition}[\protect{\cite[Definition~4.1(a)]{FJJK}}]\label{tapas}
Given $A \in 2^{\Omega}$ such that $\pi(A) = \SI$ we define the
functions
\begin{align*}
M_{_A}(\theta) &:= \max\set{x \in \I}{(\theta,x) \in A} \\
m_{_A}(\theta) &:= \min\set{x \in \I}{(\theta,x) \in A}.
\end{align*}
It can be seen that
$M_{_A}$ is an upper semicontinuous function from $\SI$ to $\I$ and
$m_{_A}$ is a lower semicontinuous function from $\SI$ to $\I.$
From \cite[Theorem~7.10]{Choq},
each of the functions $m_{_{A}}$ and $M_{_{A}}$ is continuous on a residual
set of $\SI.$
We denote by $\rescont[m_{_{A}}]$ (respectively $\rescont[M_{_{A}}]$) the
residual set of continuity of $m_{_{A}}$ (respectively $M_{_{A}}$).
\end{definition}

\begin{remark}\label{ResConPC}
If $A$ is a pseudo-curve, it follows from \cite[Lemma~7.2]{Nadler} that
$
\rescont[A] = \rescont[M_{_{A}}] = \rescont[m_{_{A}}] = \set{\theta \in \SI}{M_{_A}(\theta) = m_{_A}(\theta)}
$
(that is, $A$ is pinched in $\rescont[A] = \rescont[M_{_{A}}] = \rescont[m_{_{A}}]$) and, hence,
\[
A = \overline{\Graph\left(M_{_{A}}\evalat{\rescont[M_{_{A}}]}\right)}
  = \overline{\Graph\left(m_{_{A}}\evalat{\rescont[m_{_{A}}]}\right)}.
\]
\end{remark}

\begin{definition}[\protect{\cite[Definition~3.13]{FJJK}}]\label{orden}
Given two strips $A$ and $B$ we write $A < B$ (respectively $A \le B$)
if there exists a residual set $G$ in $\SI$ such that
$M_{_{A}}(\theta) < m_{_{B}}(\theta)$
(respectively $M_{_{A}}(\theta) \le m_{_{B}}(\theta)$)
for all $\theta \in G.$
We say that two strips are
\emph{ordered} (respectively \emph{weakly ordered})
if either $A < B$ or $A > B$
(respectively $A \le B$ or $A \ge B$).
\end{definition}

\begin{remark}\label{orderedimpliesPDInt}
It follows from the definition that
two (weakly) ordered strips have pairwise disjoint interiors.
\end{remark}

The above ordering can be better formulated in terms of the \emph{covers} of a strip.

\begin{definition}\label{plusminus}
Let $A \subset \Omega$ be a strip.
We define the \emph{top cover of $A$} as the pseudo-curve defined by
$M_{_A}\evalat{\rescont[M_{_{A}}]}$:
\[
A^{+} := \overline{\Graph\left(M_{_A}\evalat{\rescont[M_{_{A}}]}\right)},
\]
and the \emph{bottom cover of $A$} as the pseudo-curve defined by
$m_{_A}\evalat{\rescont[m_{_{A}}]}$:
\[
A^{-} := \overline{\Graph\left(m_{_A}\evalat{\rescont[m_{_{A}}]}\right)}.
\]
\end{definition}

\begin{remark}\label{m-M-func}
The sets $A^{+}$ and $A^{-}$ are bands but in general do not coincide with
$\overline{\Graph\left(M_{_A}\right)}$ and $\overline{\Graph\left(m_{_A}\right)}$
respectively.
Moreover, if $A$ is a pseudo-curve then, from Remark~\ref{ResConPC}, $A^{+} = A^{-} = A.$
\end{remark}

\begin{remark}\label{ordentapas}
Let $A$ and $B$ be two strips.
By Remark~\ref{ResConPC} we have,
$A < B$ if and only if $A^{+} < B^{-}$ and
$A \le B$ if and only if $A^{+} \le B^{-}.$
\end{remark}

To end this subsection we introduce the useful notion of \emph{band between two pseudo-curves}.
Although this definition is inspired in the definition of a \emph{basic strip} from \cite{FJJK}
(see Definition~\ref{Markovsignedstrips}) we follow our approach based in pseudo-curves.

\newcommand{\rescontAB}{\rescont[M_{_A}] \cap \rescont[m_{_B}]}
\begin{definition}
Let $A$ and $B$ be pseudo-curves such that $A < B$. We define the \emph{band between $A$ and $B$} as:
\[
E_{_{AB}} := \overline{\bigcup_{\theta\in \rescontAB} \{\theta\} \times \left(M_{_{A}}(\theta), m_{_{B}}(\theta)\right)}.
\]
\end{definition}

The properties of the set $E_{_{AB}}$ are summarized by:

\begin{lemma}\label{GoodBands}
Let $A$ and $B$ be pseudo-curves such that $A < B$. Then,
\begin{enumerate}[(a)]
\item $E_{_{AB}}^{-} = A$ and $E_{_{AB}}^{+} = B.$
Moreover,
$\left(E_{_{AB}}\right)^{\theta} = \{\theta\} \times \left[M_{_{A}}(\theta), m_{_{B}}(\theta)\right]$
for every $\theta \in \rescontAB.$
\item $E_{_{AB}}$ is a band.
\item $E_{_{AB}} := \overline{\Int\left(E_{_{AB}}\right)}.$
In particular, $E_{_{AB}}$ has non-empty interior.
\end{enumerate}
\end{lemma}

\begin{proof}
From the definition of $E_{_{AB}}$ it follows that
\[
\Graph\left(M_{_A}\evalat{\rescontAB}\right)
\subset E_{_{AB}}.
\]
Thus,
\[
A = A^c = \overline{\Graph\left(M_{_A}\evalat{\rescontAB}\right)} \subset E_{_{AB}}
\]
by Remarks~\ref{pseudocurves-properties}(1) and \ref{ResConPC} and Lemma~\ref{CoreForResidualOfContinuity}.
Consequently, $m_{_{E_{_{AB}}}} \le m_{_{A}}.$
Now we will prove that $m_{_{E_{_{AB}}}} \ge m_{_{A}}$ and, hence, $m_{_{E_{_{AB}}}} = m_{_{A}}.$
To see this note that, for every $\theta \in \SI,$ there exists a sequence
\[
 \{(\theta_n,x_n)\}_{n\in \N} \subset
 \bigcup_{\theta\in \rescontAB} \{\theta\} \times \left(M_{_{A}}(\theta), m_{_{B}}(\theta)\right)
\]
which converges to $(\theta, m_{_{E_{_{AB}}}}(\theta)).$
Observe that $x_n \ge M_{_{A}}(\theta_n) \ge m_{_{A}}(\theta_n)$ for every $n$.
Therefore, by Remark~\ref{ResConPC}
$
m_{_{E_{_{AB}}}}(\theta) = \lim_{n} x_n \ge \liminf_n m_{_{A}}(\theta_n) \ge m_A(\theta).
$

Since $m_{_{E_{_{AB}}}} = m_{_{A}}$, from Definition~\ref{plusminus} and Remark~\ref{ResConPC}
it follows that $E_{_{AB}}^{-} = A.$

In a similar way we get that $M_{_{E_{_{AB}}}} = M_{_{B}}$ and $E_{_{AB}}^{+} = B.$

Then, by the part already proven and Remark~\ref{ResConPC},
\begin{equation}\label{specialintervalsinEAB}
\left(E_{_{AB}}\right)^{\theta} = \{\theta\} \times \left[M_{_{A}}(\theta), m_{_{B}}(\theta)\right]\quad
\text{for every $\theta \in \rescontAB.$}
\end{equation}
This ends the proof of (a).

Now we prove (b).
From the previous statement it follows that $E_{_{AB}}$ is a strip.
Hence, we have to show that $\left(E_{_{AB}}\right)^c = E_{_{AB}}$
which, by Definition~\ref{core}, it reduces to prove that
$E_{_{AB}} \subset \left(E_{_{AB}}\right)^c.$
Moreover, it is enough to show that
\begin{equation}\label{corecontainswhatitshould}
E^{\theta}_{_{AB}} \subset \left(E_{_{AB}}\right)^c\quad
\text{for every $\theta \in \rescontAB$}
\end{equation}
because, by \eqref{specialintervalsinEAB},
\[
E_{_{AB}} \subset  \overline{\bigcup_{\theta \in \rescontAB} E^{\theta}_{_{AB}}}
          \subset \overline{\left(E_{_{AB}}\right)^c} = \left(E_{_{AB}}\right)^c .
\]

To prove \eqref{corecontainswhatitshould} observe that, since $\rescontAB \cap \rescont[E_{_{AB}}]$
is a residual set (contained in $\rescont[E_{_{AB}}]$), from Lemma~\ref{CoreForResidualOfContinuity} we get
\begin{equation}\label{essentialcoreofEAB}
 \left(E_{_{AB}}\right)^c =
 \overline{E_{_{AB}} \cap \pi^{-1}\left(\rescontAB \cap \rescont[E_{_{AB}}]\right)} =
  \overline{\bigcup_{\theta \in \rescontAB \cap \rescont[E_{_{AB}}]} E^{\theta}_{_{AB}}} .
\end{equation}
In particular,
\[
 \bigcup_{\theta \in \rescontAB \cap \rescont[E_{_{AB}}]} E^{\theta}_{_{AB}} \subset \left(E_{_{AB}}\right)^c .
\]

Fix $\theta \in \left(\rescontAB\right) \setminus \rescont[E_{_{AB}}].$
Since $\rescontAB \cap \rescont[E_{_{AB}}]$ is a residual set,
there exists a sequence $\{\theta_n\}_{n=1}^{\infty} \subset \rescontAB \cap \rescont[E_{_{AB}}]$
whose limit is $\theta.$ The continuity of the functions $M_{_{A}}$ and $m_{_{B}}$ in $\rescontAB$
implies that
$\lim M_{_{A}}(\theta_n) = M_{_{A}}(\theta)$
and
$\lim m_{_{B}}(\theta_n) = m_{_{B}}(\theta).$
Therefore, again by \eqref{specialintervalsinEAB},
every point of $E_{_{AB}}^{\theta}$ is limit of points in
$\{E_{_{AB}}^{\theta_n}\}_{n=1}^{\infty}.$
This implies that $E^{\theta}_{_{AB}} \subset \left(E_{_{AB}}\right)^c$ by \eqref{essentialcoreofEAB}.
This ends the proof of (b).

To prove (c) observe that
$\overline{\Int\left(E_{_{AB}}\right)} \subset E_{_{AB}}.$
So, it is enough to show that
\[
\bigcup_{\theta\in \rescontAB} \{\theta\} \times \left(M_{_{A}}(\theta), m_{_{B}}(\theta)\right)
\subset \Int\left(E_{_{AB}}\right).
\]

Take
$(\theta, x) \in \{\theta\} \times \left(M_{_{A}}(\theta), m_{_{B}}(\theta)\right)$
with $\theta \in \rescontAB.$
Since $x \ne M_{_{A}}(\theta)$ and $x \ne m_{_{B}}(\theta),$
there exists $\varepsilon > 0$ such that
$x > M_{_{A}}(\theta) + \varepsilon$ and $x < m_{_{B}}(\theta) - \varepsilon.$
On the other hand, the continuity of $M_{_{A}}$ and $m_{_{B}}$ on
$\rescontAB$
implies that there exist $\delta > 0$ such that
$\theta' \in \rescontAB$ and  $|\theta - \theta'| < \delta$ implies
$|M_{_{A}}(\theta) - M_{_{A}}(\theta')| < \varepsilon$ and $|m_{_{B}}(\theta) - m_{_{B}}(\theta')| < \varepsilon.$
Now we define
\[
 r := \min \left\{\delta, |x-M_{_{A}}(\theta)-\varepsilon|, |x-m_{_{B}}(\theta)+ \varepsilon| \right\} > 0.
\]
Observe that, with this choice of $r,$ $M_{_{A}}(\theta) + \varepsilon \le x -r < x+r \le m_{_{B}}(\theta) - \varepsilon.$

Let
$
U := \set{(\theta', y) \in \Omega}{|\theta - \theta'| < r \text{ and }  |x-y| < r}
$
be an open neighbourhood of $(\theta, x).$
We will prove that every $(\theta', y) \in U$ belongs to $E_{_{AB}}.$
If $\theta' \in \rescontAB,$
from the choice of $\delta$ and $r,$ it follows that
$
(\theta', y) \in \{\theta'\} \times (x-r,x+r) \subset \{\theta'\} \times \left[M_{_{A}}(\theta'), m_{_{B}}(\theta')\right] \subset E_{_{AB}}.
$
Now assume that $\theta' \notin \rescontAB$
and consider a sequence
$\{\theta_n\}_{n \in \N} \subset \rescontAB \cap (\theta-r, \theta+r)$
converging to $\theta'.$
Clearly, $(\theta_n, y) \in U$ for every $n\in \N$ and,
by the part already proven, $(\theta_n, y) \in E_{_{AB}}.$
Consequently, since $E_{_{AB}}$ is closed,
$(\theta', y) = \lim (\theta_n, y) \in E_{_{AB}}.$
\end{proof}

\subsection{Strip patterns}

In this subsection we define the notion of strips pattern and forcing
for maps from $\cSO$ along the lines of Subsection~\ref{IntPat}.

\begin{definition}[\protect{\cite[Definition~3.15]{FJJK}}]\label{def-ban}
Let $F \in \cSO.$ We say that a strip $A \subseteq \Omega$
is a \emph{$p$-periodic strip}
if $F^{p}(A) = A$ and the strips $A,F(A),\ldots, F^{p-1}(A)$
are pairwise disjoint and ordered.
The set $\{A,F(A),\ldots, F^{p-1}(A)\}$ is called an
\emph{$n$-periodic orbit of strips}.

By Remarks~\ref{band-properties} and \ref{solidorpseudocurve}, it follows that we can restrict our attention to
two kind of periodic orbit of bands: the solid ones and the pseudo-curves.
\end{definition}

A periodic orbit of strips $\{B_1,B_2,\ldots, B_p\}$
is said to have the \emph{spatial labelling} if
$B_1 < B_2 < \ldots < B_p.$
In what follows we will assume that every periodic orbit of strips
has the spatial labelling.

\begin{definition}[Strip pattern]
Let $F \in \cSO$ and let $\cB = \{B_1,B_2,\dots,B_n\}$ be a
periodic orbit of strips.
The \emph{strips pattern of $\cB$} is the permutation $\tau$ such that
$F(B_{i}) = B_{\tau(i)}$ for every $i = 1,2,\dotsc,n.$

When a map $F \in \cSO$ has a periodic orbit of strips with
strips pattern $\tau$ we say that \emph{$F$ exhibits the pattern $\tau$.}
\end{definition}

\begin{remark}
Interval and strips patterns are formally the same algebraic objects;
that is cyclic permutations.
\end{remark}

\begin{definition}[Forcing]
Let $\tau$ and $\nu$ be strips patterns.
We say that $\tau$ forces $\nu$ in $\Omega,$
denoted by $\tau \Longrightarrow_{\Omega} \nu,$
if and only if every map $F \in \cSO$
that exhibits the strips pattern~$\tau$ also exhibits the quasiperiodic
pattern~$\nu$.
\end{definition}

The next theorem is the first main result of this paper.
It characterizes the relation $\Longrightarrow_{\Omega}$
by comparison with $\Longrightarrow_{\I}.$

\begin{MainTheorem}\label{teo-prin-A}
Let $\tau$ and $\nu$ be patterns (both in $\I$ and $\Omega$).
Then,
\[
\tau\Longrightarrow_{\I}\nu
  \quad\text{if and only if}\quad
\tau\Longrightarrow_{\Omega}\nu.
\]
\end{MainTheorem}

The first important consequence of Theorem~\ref{teo-prin-A} is the
next result which follows from the fact that the Sharkovski\u{\i}
theorem is a corollary of the forcing relation for interval maps.

\begin{corollary}
The Sharkovski\u{\i} Theorem for maps from $\cSO$ holds.
\end{corollary}

\begin{proof}
Assume that $F \in \cSO$ exhibits a $p$-periodic strips pattern
$\tau$ and let $q \in \N$ be such that $p \gtso{\Sho} q.$
By \cite[Corollary~2.7.4]{ALM}, $\tau \Longrightarrow_{\I} \nu$ for
some strips pattern $\nu$ of period $q$. Then, by
Theorem~\ref{teo-prin-A}, $\tau \Longrightarrow_{\Omega} \nu$ and, by
definition, $F$ also has a $q$-periodic orbit of strips (with strips pattern
$\nu$). Then the corollary follows from Remark~\ref{band-properties}(2,3).
\end{proof}

Next we are going to study the relation between the forcing relation
and the topological entropy of maps from $\cSO$. To this end we
introduce the notion of horseshoe in $\cSO.$

Let $F \in \cSO$ and let $A$ and $B$ be bands in $\Omega.$
We say that $A$ \emph{$F$-covers} $B$ if either
$F(A^{-}) \le B^{-}$ and $F(A^{+}) \ge B^{+},$ or
$F(A^{-}) \ge B^{+}$ and $F(A^{+}) \le B^{-}.$

\begin{definition}[Horseshoe]
An \emph{$s$-horseshoe} for a map $F \in \cSO$ is a pair
$(J,\mathcal{D})$ where $J$~is a band and
$\mathcal{D}$ is a set of $s \ge 2$ pairwise weakly ordered bands,
each of them with non-empty interior,
such that $L$ $F-$covers $J$ for every $L \in \mathcal{D}.$
Observe that, by Remark~\ref{orderedimpliesPDInt},
the elements of $\mathcal{D}$ have pairwise disjoint interiors.
\end{definition}

The next theorem is the second main result of the paper. It
relates the topological entropy of maps from $\cSO$ with
horseshoes.

\begin{MainTheorem}\label{teo-prin-B}
Assume that $F \in \cSO$ has an $s$-horseshoe. Then
\[
h(F) \ge \log(s).
\]
\end{MainTheorem}

Next we want to introduce a class of maps that play the role of the
connect-the-dots maps in the interval case and use them to study the
topological entropy in relation with the periodic orbits of strips.

\begin{definition}[Quasiperiodic $\tau$-linear map]
Given a strips pattern $\tau$ we define a \emph{quasiperiodic
$\tau$-linear map} $F_{\tau} \in \cSO$ as:
\[
F_{\tau} (\theta, x) := (R_{\omega}(\theta), f_{\tau}(x))
\]
where $R_\omega$ is the irrational rotation by angle $\omega$ and
$f_{\tau}$ is a $\tau$-linear interval map
(Definition~\ref{connect-the-dots} --- recall that $\tau$ is also an
interval pattern).
\end{definition}

\begin{remark}\label{per-ban}
Since, by definition, $f_{\tau}$ has a periodic orbit with interval
pattern $\tau,$ $F_{\tau}$ has a periodic orbit of bands
(in fact curves which are horizontal circles) with strips pattern
$\tau.$
\end{remark}

The next main result shows that the quasiperiodic
$\tau$-linear maps have minimal entropy among all maps from $\cSO$
which exhibit the strips pattern $\tau$, again as in the interval case.

\begin{MainTheorem}\label{teo-prin-C}
Assume that $F \in \cSO$ exhibits the strips pattern~$\tau$. Then
\[
h(F) \ge h (F_{\tau}) = h (f_{\tau}).
\]
\end{MainTheorem}

Theorem~\ref{teo-prin-C} has an interesting consequence concerning
the entropy of strips patterns that we define as follows.

\begin{definition}[Entropy of strips patterns]
Given a strips pattern $\tau$ we define the \emph{entropy of $\tau$} as
\[
h(\tau) := \inf\set{h(F)}{\text{$F \in \cSO$ and $F$ exhibits
    the strips pattern $\tau$}}.
\]
\end{definition}

With this definition, in view of the Remark~\ref{per-ban},
Theorem~\ref{teo-prin-C} can be written as follows:

\setcounter{MainTheorem}{2}
\begin{MainTheorem}
Assume that $F \in \cSO$ exhibits the strips pattern~$\tau$. Then
\[
h(\tau) = h (F_{\tau}) = h (f_{\tau}).
\]
\end{MainTheorem}

By \cite[Corollary~4.4.7]{ALM} and \cite[Lemma~4.4.11]{ALM} we
immediately get the following simple but important corollary of
Theorem~\ref{teo-prin-C} which will allow us to obtain lower bounds of
the topological entropy depending on the set of periods.

\begin{corollary}
Assume that $\tau$ and $\nu$ are strips patterns such that
$\tau \Longrightarrow_{\Omega} \nu.$
Then $h(\tau) \ge h(\nu).$
\end{corollary}

\begin{corollary}
Assume that $F \in \cSO$ has a periodic orbit of strips of
period~$2^{n}q$ with $n \ge 0$ and $q \ge 1$ odd. Then,
\[
h(F) \ge \frac{\log(\lambda_q)}{2^n}
\]
where $\lambda_1 = 1$ and, for each $q \ge 3$ odd, $\lambda_q$ is the
largest root of the polynomial $x^q - 2 x^{q-2} - 1$.
Moreover,
for every $m=2^n q$ with $n \ge 0$ and $q \ge 1$ odd,
there exists a map $F_m \in \cSO$ with a periodic orbit of bands of period $m$
such that $h(F_m) = \tfrac{\log(\lambda_q)}{2^n}.$
\end{corollary}

\begin{proof}
Let $\tau$ denote the strips pattern of a periodic orbit of strips of
$F$ of period~$2^{n}q.$ By Theorem~\ref{teo-prin-C} and \cite{BGMY}
(see also Corollaries~4.4.7 and 4.4.18 of \cite{ALM}) we get that
\[
h(F) \ge h(f_{\tau}) \ge \frac{\log \lambda_q}{2^{n}}.
\]
To prove the second statement we use  \cite[Theorem~4.4.17]{ALM}:
for every $m=2^n q$ there exists a primary pattern $\nu_m$ of period $m$ such that
$h(f_{\nu_m}) = \frac{\log \lambda_q}{2^{n}}.$
Then, from Theorem~\ref{teo-prin-C}, we can take $F_m = F_{\nu_m}$.
\end{proof}

\section{Proof of Theorem~\ref{teo-prin-A}}\label{ProofOfteo-prin-A}

To prove Theorem~\ref{teo-prin-A} we need some more notation and
preliminary results.

An important tool in the study of patterns is the Markov graph.
Signed Markov graphs are a specialization of Markov graphs.
Next we define them and clarify the relation with our situation.

A a \emph{combinatorial (directed) signed graph} is defined as a
pair $G = (V, \mathcal{A})$ where
$V$ is a finite set, called the \emph{set of vertices}, and
$\mathcal{A} \subset V \times V \times \{+, -\}$ is called the \emph{set of signed arrows}.
Given a \emph{signed arrow} $\alpha = (I, J, s) \in \mathcal{A},$
$I$ is the \emph{beginning of $\alpha$}, $J$ is the \emph{end of
$\alpha$} and $s$ is the \emph{sign of $\alpha$}.
Such an arrow $\alpha$ is denoted by $I \signedarrow{s} J.$

\subsection{Signed Markov graphs in the interval}

We start by introducing the notion of \emph{signed covering}.
In what follows, $\Bd(A)$ will denote the boundary of $A.$

\begin{definition}\label{sig-arrow}
Let $f \in \cI$ and let $I,J \subset \I$ be two intervals.
We say that $I$ \emph{positively $F$-covers} $J$, denoted by
$I \signedarrow{+} J$
(or $I \signedarrow[f]{+} J$ if we need to specify the map),
if $f(\min I) \le \min J < \max J \le f(\max I)$ and,
analogously, we say that $I$ \emph{negatively $F$-covers} $J$,
denoted by $I \signedarrow{-} J$ (or $I \signedarrow[f]{-} J$),
if $f(\max I) \le \min J < \max J \le f(\min I).$
Observe that if $I \signedarrow{s_{1}} J_1$ and
$I \signedarrow{s_{2}} J_2$ then $s_{1} = s_{2}.$

We will write $I \signedarrowequal{s_{1}} J$ or $I \signedarrowequal[f]{s_{1}} J$
to denote that $f(I) = J$ and $I \signedarrow[f]{s_{1}} J$
(in particular, $f(\Bd(I)) = \Bd(J)$).
\end{definition}

We associate a signed graph to a periodic orbit of an interval map as
follows.
\begin{definition}\label{Markovsignedinterval}
Let $f \in \cI$ and let $P$ be a periodic orbit $f$.
A \emph{$P$-basic interval} is the closure of a connected
component of $[\min P,\max P] \setminus P.$
The \emph{$P$-signed Markov graph of $f$} is the combinatorial signed
graph that has the set of all basic intervals as set of vertices
$V$ and the signed arrows in $\mathcal{A}$ are the ones given by
Definition~\ref{sig-arrow}.
\end{definition}

\begin{remark}\label{graphunique-int}
Observe that the $P$-signed Markov graph of $f$ depends only on
$f\evalat{P}$ or more precisely on the pattern of $P$. It does not
depend on the concrete choice of the points of $P$ and on the graph
of $f$ outside $P$. Consequently, if
$P$ is a periodic orbit of $f \in \cI$ and
$Q$ is a periodic orbit of $g \in \cI$
with the same pattern then the $P$-signed Markov graph of
$f$ and the $Q$-signed Markov graph of $g$ coincide.
In particular, the $P$-signed Markov graph of
$f$ and the $P$-signed Markov graph of $f_P$ coincide.
\end{remark}

\subsection{Signed Markov graphs in $\Omega$}

Now we also associate a signed graph to a periodic orbit of strips.
We start by defining the notion of \emph{signed covering} for bands.
It is an improvement of the notion of $F$-covering introduced before.

\begin{definition}[Signed covering \protect{\cite[Definition~4.14]{FJJK}}]\label{signedcovering}
Let $F \in \cSO$ and let $A$ and $B$ be bands in $\Omega.$
We say that $A$ \emph{positively $F$-covers} $B$, denoted by
$A \signedarrow{+} B$
(or $A \signedarrow[F]{+} B$ if we need to specify the map),
if\footnote{%
Although these definitions are formally different from \cite[Definition~4.14]{FJJK},
they are equivalent by \cite[Lemma~4.3(c,d)]{FJJK} and the definitions of the
weak ordering of strips.}\label{footremark}
$F(A^{-}) \le B^{-}$ and $F(A^{+}) \ge B^{+}$
and, analogously, we say that $A$ \emph{negatively $F$-covers} $B$, denoted by
$A \signedarrow{-} B$ (or $A \signedarrow[F]{-} B$),
if $F(A^{-}) \ge B^{+}$ and $F(A^{+}) \le B^{-}.$

Observe that, as in the interval case (see Definition~\ref{sig-arrow}), if
$A \signedarrow{s_1} B_1$ and $A \signedarrow{s_2} B_2,$ then
$s_1 = s_2.$

We will write $A \signedarrowequal{s_{1}} B$ or $A \signedarrowequal[F]{s_{1}} B$
to denote that $F(A) = B$ and $A \signedarrow[F]{s_{1}} B.$
\end{definition}

Next, by using the notion of band between two pseudo-curves,
we will define the analogous of basic interval (\emph{basic band}) and
signed Markov graph for maps from $\cSO$.

\begin{definition}\label{Markovsignedstrips}
Let $F \in \cSO$ and let $\mathcal{B} = \{B_{1},B_{2},\ldots,B_{n}\}$ be
a periodic orbit of strips of $F$ with the spatial labelling
(that is, $B_{1} < B_{2} < \dots < B_{n}$).
For every $i = 1,2,\dots, n-1$ the band
(see Remark~\ref{ordentapas} and Lemma~\ref{GoodBands})
\[
I_{_{B_{i}B_{i+1}}} := E_{_{B_{i}^{+}B_{i+1}^{-}}}
   = \overline{\Int\left(E_{_{B_{i}^{+}B_{i+1}^{-}}}\right)}
\]
will be called a \emph{basic band}.
Observe that, from Lemma~\ref{GoodBands}(a),
$I_{_{B_{i}B_{i+1}}}^{-} = B_{i}^{+}$ and $I_{_{B_{i}B_{i+1}}}^{+} = B_{i+1}^{-}.$

The \emph{$\mathcal{B}$-signed Markov graph of $F$} is the combinatorial
signed graph that has the set of all basic bands as set of vertices
$V$ and the signed arrows in $\mathcal{A}$ are the ones given by
Definition~\ref{signedcovering}.
\end{definition}

Clearly, all the basic bands are contained in $E_{_{B_{1}B_{n}}},$
$I_{_{B_{i}B_{i+1}}} \le I_{_{B_{i+1}B_{i+2}}}$ for $i=1,2,\dots, n-2$
and if
$I_{_{B_{i}B_{i+1}}} \cap I_{_{B_{j}B_{j+1}}} \ne \emptyset$ then
$|i-j| = 1.$

\begin{remark}\label{graphunique-strips}
As in the interval case (see Remark~\ref{graphunique-int})
the $P$-signed Markov graph of $F$ is a pattern invariant.
Moreover, if
$P$ is a periodic orbit of $F \in \cSO$ and
$Q$ is a periodic orbit of the interval map $f \in \cI$
with the same pattern, then the $P$-signed Markov graph of
$F$ and the $Q$-signed Markov graph of $f$ coincide.
In particular, the $P$-signed Markov graph of
$F$ and the $P$-signed Markov graph of $f_P$ coincide.
\end{remark}

The following lemma summarizes the properties of basic bands and
arrows. We will use it in the proof of Theorem~\ref{teo-prin-A}.

\begin{lemma}\label{prop-band}
The following statements hold.
\begin{enumerate}[(a)]
\item Let $F \in \cSO$ and let $A$ and $B$ be bands such that
there is a signed arrow $A \signedarrow{s} B$ from $A$ to $B$ in the
signed Markov graph of $F$. Then,
\begin{enumerate}[({a}.1)]
\item $F(A) \supset B$.
\item $A \signedarrow{s} D$ for every band $D \subset B.$
\item There exists a band $C \subset A$ such that $C \signedarrowequal{s} B.$
      Moreover, $F(C^+) \subset B^+$ and $F(C^-) \subset B^-$ if $s = +,$
      and       $F(C^-) \subset B^+$ and $F(C^+) \subset B^-$ if $s = -.$
\item Assume that $A \signedarrow{s} \widetilde{B}$ with
$B \le \widetilde{B}$ and let $C$ and $\widetilde{C}$ denote the
bands given by (a.3) for $B$ and $\widetilde{B}$ respectively.
Then,
$C \le \widetilde{C}$ if $s = +,$ and
$C \ge \widetilde{C}$ if $s = -.$
\end{enumerate}
\item Let $F \in \cSO$ and let $A$ be a band such that
$A \signedarrow{\pm} A.$
Then there exists a band $A_{\infty} \subset A$ such that
$A_{\infty} \signedarrowequal{\pm} A_{\infty}.$
\end{enumerate}
\end{lemma}

\begin{proof}
Statement (a.1) is \cite[Lemma~4.15]{FJJK} and (a.2)
follows directly from the definitions.
Statements (a.3,4) are \cite[Lemma~4.19]{FJJK} while
statement (b) is \cite[Lemma~4.21]{FJJK}.
\end{proof}

\subsection{Loops of signed Markov graphs}

Given a combinatorial signed Markov graph $G$, a sequence of arrows
$
\alpha = I_0 \signedarrow{s_{0}} I_1 \signedarrow{s_{1}} \cdots \signedarrow{s_{m-1}} I_{m-1}
$
will be called a \emph{path of length $m$}.
The length of $\alpha$ will be denoted by $\abs{\alpha}$.
When a path begins and ends in the same vertex (i.e. $I_{m-1} =
I_{0}$) it will be called a \emph{loop}.
Observe that, then
$
I_1 \signedarrow{s_{1}} I_2 \signedarrow{s_{2}} \cdots \signedarrow{s_{m-2}} I_{m-1} \signedarrow{s_{m}} I_{1}
$
is also a loop in $G.$
This loop is called a \emph{shift} of $\alpha$ and denoted by
$S(\alpha).$
For $n \ge 0,$ we will denote by $S^{n}$ the $n$-th iterate of
the shift. That is,
\[
S^n(\alpha) = I_{j_0} \signedarrow{s_{j_0 }} I_{j_1}
                      \signedarrow{s_{j_1}} I_{j_2} \signedarrow{s_{j_2}}  \cdots
                      \signedarrow{s_{j_{m-2}}} I_{j_{m-1}},
\]
where $j_r = r + n \pmod{m}.$
Note that $S^{km}(\alpha) = \alpha$ for every $k \ge 0.$

Let
$
\alpha = I_0 \signedarrow{s_{0}} I_1 \signedarrow{s_{1}} \cdots \signedarrow{s_{m-2}} I_{m-1}
$
and
$
\beta = J_0 \signedarrow{r_{0}} J_1 \signedarrow{r_{1}} \cdots \signedarrow{r_{l-2}} J_{l-1}
$
be two paths such that the last vertex of $\alpha$ coincides with
the first vertex of $\beta$ (i.e. $I_{m-1} = J_0$).
The path
$
I_0 \signedarrow{s_{0}} I_1 \signedarrow{s_{1}} \cdots \signedarrow{s_{m-1}} J_0
    \signedarrow{r_{0}} J_1 \signedarrow{r_{1}} \cdots \signedarrow{r_{l-1}} J_{l-1}
$
is the \emph{concatenation} of $\alpha$ and $\beta$ and is denoted by
$\alpha\beta.$
In this spirit, for every $n \ge 1,$ $\alpha^n$ will denote the
concatenation of $\alpha$ with himself $n$-times.
the path $\alpha^n$ will be called the \emph{$n$-repetition of
$\alpha$}. Also, $\alpha^\infty$ will denote the infinite path
$\alpha\alpha\alpha\cdots$.

A loop is called \emph{simple} if it is not a repetition of a shorter
loop. Observe that, in that case, the length of the shorter loop
divides the length of the long one.

The next lemma translates the non-repetitiveness of a loop to conditions on its liftings.
Its proof is folk knowledge.

\begin{lemma}\label{simpleisdifferent}
Let $\alpha$ be a signed loop of length $n$ in a combinatorial signed Markov graph $G$.
If $\alpha$ is simple, $S^i(\alpha) \ne S^j(\alpha)$ for every $i \ne j.$
\end{lemma}

Given a path
$
\alpha = I_0 \signedarrow{s_0} I_1 \signedarrow{s_1} \ldots I_{m-1} \signedarrow{s_{m-1}} I_{m}
$
we define the \emph{sign} of $\alpha$, denoted by $\Sign(\alpha),$
as $\prod_{i = 1}^{m} s_i,$
where in this expression we use the obvious multiplication rules:
\begin{align*}
& + \cdot + = − \cdot − = +,\text{ and}\\
& + \cdot − = − \cdot + = −.
\end{align*}

Finally we introduce a (lexicographical) \emph{ordering} in the set
of paths of signed combinatorial graphs.
To this end we start by introducing a linear ordering in the set of
vertices. This ordering is arbitrary but fixed.

In the case of Markov graphs, the spatial labelling of orbits induces
a natural ordering in the set of basic intervals or basic bands,
which is the ordering that we are going to adopt.
More precisely, if $P = \{p_0, p_1,\dots,p_{n-1}\}$ is a periodic
orbit with the spatial labelling, then we endow the set of vertices
(basic intervals) of the associated signed Markov graph with the
following ordering:
\[
 [p_0,p_1] < [p_1,p_2] < \dots < [p_{n-2},p_{n-1}].
\]
Analogously, if $\cB = \{B_0, B_1,\dots,B_{n-1}\}$ is a periodic orbit
of strips with the spatial labelling, then we endow the set of
vertices (basic intervals) of the associated signed Markov graph with
the following  ordering:
\[
I_{_{B_{0}B_{1}}} < I_{_{B_{1}B_{2}}} < \dots < I_{_{B_{n-2}B_{n-1}}}.
\]

Then, the above ordering in the set of vertices naturally induces
a \emph{lexicographical ordering} in the set of paths of the
signed combinatorial graph as follows.
Let
\begin{align*}
\alpha & = I_0 \signedarrow{s_{0}}
         I_1 \signedarrow{s_{1}} \cdots
         I_{n-1} \signedarrow{s_{n-1}} I_{n}
\ \text{and}\\
\beta &= J_0 \signedarrow{r_{0}}
         J_1 \signedarrow{r_{1}} \cdots
         J_{m-1} \signedarrow{r_{m-1}} J_{m}
\end{align*}
be paths such that there exists $k \le \min\{n,m\}$
with $I_k \ne J_k$ and $I_i = J_i$ for $i=0,1,\dots,k-1$
(recall that, by Definition~\ref{sig-arrow},
if $I_i = J_i$ then the signs $s_i$ and $r_i$ of the
corresponding arrows coincide).
We write $\alpha < \beta$ if and only if
\[
\begin{cases}
I_k < J_k & \text{when $s = +$, or} \\
I_k > J_k & \text{when $s = -$,}
\end{cases}
\]
where $s = \Sign\left(
  I_0 \signedarrow{s_{0}}
  I_1 \signedarrow{s_{1}} \cdots
  I_{k-1} \signedarrow{s_{k-1}} I_{k}
\right) = s_0s_1\cdots s_{k-1}.$

Next we relate the loops in signed Markov graphs with periodic orbits.

\begin{definition}\label{asoc-int}
Let $f \in \cI$ and let $p$ be a periodic point of $f$
and let
\[
\alpha = J_0 \signedarrow{s_{0}} J_1 \signedarrow{s_{1}} \cdots
J_{n-1} \signedarrow{s_{n-1}} J_{0}
\]
be a loop in the $P$-signed Markov graph of $f.$
We say that $\alpha$ and $p$ are \emph{associated}
if $p$ has period $n$ and $f^{i}(p) \in J_{i}$ for every $i = 0,1, \ldots, n-1.$
Observe that in such case $S^{m}(\alpha)$ and $f^m(p)$ are associated for all $m \ge 1.$
\end{definition}

The next lemma relates the ordering of periodic points with the ordering of the associated loops.
Its proof is a simple exercise.

\begin{lemma}\label{lazosintervalasociados}
Let $f \in \cI$ and let $f_P$ be a $P$-linear map, where $P$
is a periodic orbit. Let $x$ and $y$ be two distinct periodic points
of $f_P$ associated respectively to two distinct loops $\alpha$ and $\beta$ in
the $P$-signed Markov graph of $f_P.$
Then $x < y$ if and only if $\alpha < \beta.$
Consequently, for every $n \ge 1,$
$f^n(x) < f^n(y)$ if and only if $S^n(\alpha) < S^n(\beta).$
\end{lemma}

% \begin{proof}
% Set
% \begin{align*}
% \alpha & = I_0 \signedarrow{s_{0}}
%          I_1 \signedarrow{s_{1}} \cdots
%          I_{k-1} \signedarrow{s_{k-1}}
%          I_{k} \signedarrow{s_{k}} I_{k+1} \cdots
% \ \text{and}\\
% \beta &= I_0 \signedarrow{s_{0}}
%          I_1 \signedarrow{s_{1}} \cdots
%          I_{k-1} \signedarrow{s_{k-1}}
%          J_{k} \signedarrow{r_{k}} J_{k+1} \cdots
% \end{align*}
% with $I_k \ne J_k$
% Since $x$ is associated to $\alpha$ and $y$ to $\beta$ we have that
% $f_P^i(x), f_P^i(y) \in I_i$ for $i=0,1,\dots, k-1,$
% $f_P^k(x) \in I_k$ and $f_P^k(y) \in J_k.$
% Moreover, $f_P\evalat{\chull{f_P^i(x), f_P^i(y)}}$
% is increasing (respectively decreasing)
% if $s_i = +$ (respectively $s_i = -$) for  $i=0,1,\dots, k-1.$
% Consequently, $f_P^k\evalat{\chull{x,y}}$ is increasing
% (respectively decreasing) if
% $
% s = \Sign\left(
%   I_0 \signedarrow{s_{0}}
%   I_1 \signedarrow{s_{1}} \cdots
%   I_{k-1} \signedarrow{s_{k-1}} I_{k}
% \right) = s_0s_1\cdots s_{k-1} = +
% $
% (respectively $s = -$).
% Therefore, $x < y$ is equivalent to
% $f_P^i(x) < f_P^i(y)$ when $s = +$ and
% $f_P^i(x) > f_P^i(y)$ when $s = -$ which, in turn, is equivalent to
% $I_k < J_k$ when $s = +$ and
% $I_k > J_k$ when $s = -.$
% This ends the proof of the first statement of the lemma. The second
% one is a direct consequence of it and Definition~\ref{asoc-int}.
% \end{proof}

The next lemma is folk knowledge but we include the proof
because we are not able to provide an explicit reference for it.

\begin{lemma}\label{lazo-unico}
Let $\tau$ be a pattern and let $f_{\tau} = f_P$ be a $P$-linear
map, where $P$ is a periodic orbit of $f_P$ of pattern $\tau.$
Assume that $\{q_0,q_1,q_2,\ldots,q_m\}$ is a periodic orbit of
$f_{\tau}$ with pattern $\nu \ne \tau.$ Then there exists a unique
loop $\alpha$ in the $P$-signed Markov graph of $f_P$ associated to
$q_0.$
Moreover, $\alpha$ is simple.
\end{lemma}

\begin{proof}
The existence and unicity of the loop $\alpha$ follows from
\cite[Lemma~1.2.12]{ALM}. We have to show that $\alpha$
is simple.
Assume that $\alpha$ is the $k$ repetition of a loop
\[
\beta = J_0 \signedarrow{s_{0}} J_1 \signedarrow{s_{1}} \cdots
J_{\ell-1} \signedarrow{s_{\ell-1}} J_{0}
\]
of length $\ell$ with $k \ge 2$ and $m = k\ell.$
By \cite[Lemma~1.2.6]{ALM}, there exist intervals
$
K_0 \subset J_0, K_1 \subset J_1, \ldots,
K_{\ell-1} \subset J_{\ell-1}
$
such tat $K_i \signedarrowequal{s_i} K_{i+1}$ for
$i=0,1,\dots,\ell-2$ and
$K_{\ell-1} \signedarrowequal {s_{\ell-1}} J_{0}.$
Clearly, since $f_P$ is $P$-linear,
$f_P^\ell\evalat{K_0}$ is an affine map from $K_0$ onto $J_0$.
On the other hand, since $q_0$ is associated to $\alpha = \beta^k$ it
follows that
$
f_P^i(q_0),f_P^{i + \ell}(q_0),\dots,f_P^{i + (k-1)\ell}(q_0) \in J_i
$
for $i=0,1,\dots,\ell-1$ and, consequently,
$q_0,f_P^{\ell}(q_0),\dots,f_P^{(k-1)\ell}(q_0) \in K_0.$
Consequently, since
$f_P^\ell\left(f_P^{(k-1)\ell}(q_0)\right) = f_p^{m}(q_0) = q_0,$
it follows that
$\{q_0,f_P^{\ell}(q_0),\dots,f_P^{(k-1)\ell}(q_0)\}$ is a periodic
orbit $f_P^\ell\evalat{K_0}$ with period $k \ge 2$.
The affinity of $f_P^\ell\evalat{K_0}$ implies that
$f_P^\ell\evalat{K_0}$ is decreasing with slope -1 and $k = 2$.
The fact that $f_P^\ell\evalat{K_0}(K_0) = J_0$ implies that
$K_0 = J_0$ and the endpoints of $J_0$ are also a periodic orbit
of $f_P^\ell\evalat{K_0}$ of period 2. In this situation $P$ and
$\{q_0,q_1,q_2,\ldots,q_m\}$ both have the same period and pattern;
a contradiction.
\end{proof}

Now we want to extend the notion of associated periodic orbit and loop
and Lemma~\ref{lazosintervalasociados} to periodic orbits of strips.

\begin{definition}\label{asoc-band}
Let $F \in \cSO$ and let and let $\cB$ be a periodic orbit of strips of $F$.
We say that a loop
\[
\alpha = J_0 \signedarrow{s_{0}} J_1 \signedarrow{s_{1}} \cdots
J_{n-1} \signedarrow{s_{n-1}} J_{0}
\]
in the $\cB$-signed Markov graph of $F$ and a strip $A$ are
\emph{associated} if $A$ is an $n-$periodic strip of $F$ and
$F^{i}(A) \in J_{i}$ for every $i = 0,1, \ldots, n-1.$
Observe that in such case $S^{m}(\alpha)$ and $F^m(A)$ are associated
for all $m \ge 1.$
\end{definition}

The next lemma extends Lemma~1.2.7 and Corollary~1.2.8 of
\cite{ALM} to quasiperiodically forced skew products on the cylinder.

\begin{lemma}\label{lazoasociadobandas}
Let $F \in \cSO$ and let $J_{0},J_{1},\ldots,J_{n-1}$
be basic bands such that
\[
\alpha = J_0 \signedarrow{s_{0}} J_1 \signedarrow{s_{1}} \cdots
J_{n-1} \signedarrow{s_{n-1}} J_{0}
\]
is a simple loop in a signed Markov graph of $F.$
Then there exists a periodic band $C \subset J_0$
associated to $\alpha$ (and hence of period $n$).
Moreover, for every $i,j \in \{0,1,\dots,n-1\},$
$F^i(C) < F^j(C)$ if and only if $S^i(\alpha) < S^j(\alpha).$
\end{lemma}

\begin{proof}
Let $A$ be a basic band and let $B_1 \le B_2 \le \dots \le B_m$ be all
basic bands $F$-covered by $A.$
By Lemma~\ref{prop-band}(a.3,4) there exist bands
$K(A,B_1) \le K(A,B_2) \le \dots \le K(A,B_m)$ contained in $A$
such that $K(A,B_i) \signedarrowequal{s_A} B_i$ for $i=1,2,\dots,m,$
where $s_A$ denotes the sign of all arrows $A \signedarrow{s_A} B_i$
(see Definition~\ref{signedcovering}).

Now we recursively define a family of $2n$ bands in the following way.
We set $K_{2n-1} := K(J_{n-1},J_0) \subset J_{n-1}$ so that
$K_{2n-1} \signedarrowequal{s_{n-1}} J_0.$

Then, assume that $K_{j} \subset J_{j \pmod{n}}$ have already been
defined for $j=i+1,i+2,\dots, 2n-1$ and $i\in \{0,1,\dots, 2n-2\}.$
Since
$J_{\widetilde{\imath}} \signedarrow{s_{\widetilde{\imath}}} J_{i+1 \pmod{n}}$
with $\widetilde{\imath} = i \pmod{n},$ by Lemma~\ref{prop-band}(a.2,3),
there exists a band
$
K_{i} \subset K\left(J_{\widetilde{\imath}},J_{i+1 \pmod{n}}\right) \subset J_{\widetilde{\imath}}
$
such that $K_{i} \signedarrowequal{s_{\widetilde{\imath}}} K_{i+1}.$

Now we claim that for every $i,j \in \{0,1,\dots,n-1\}$
$S^i(\alpha) < S^j(\alpha)$ is equivalent to $K_i \le K_j.$
If $S^i(\alpha) \ne S^j(\alpha)$
there exists $k\in \{0,1,\dots,n-1\}$ such that
\begin{align*}
S^i(\alpha) &= J_{i} \signedarrow{s_i}
               J_{i+1} \signedarrow{s_{i+1}} \cdots
               J_{k+i-1} \signedarrow{s_{k+i-1}}
               J_{k+i} \signedarrow{s_{k+i}} J_{k+i+1} \cdots \text{ and}\\
S^j(\alpha) &= J_{i} \signedarrow{s_i}
               J_{i+1} \signedarrow{s_{i+1}} \cdots
               J_{k+i-1} \signedarrow{s_{k+i-1}}
               J_{k+j} \signedarrow{s_{k+j}} J_{k+j+1} \cdots
\end{align*}
with $J_{k+i \pmod{n}} \ne J_{k+j \pmod{n}}$
(where every sub-index in the above paths must be read modulo $n$).
By construction, $K_{k+i} \subset J_{k+i \pmod{n}}$ and $K_{k+j} \subset J_{k+j \pmod{n}}.$
Hence,
$K_{k+i} \le K_{k+j}$ if and only if $J_{k+i \pmod{n}} < J_{k+j \pmod{n}}.$
By definition
\[
K_{k+i-1} \signedarrowequal{s_{k+i-1 \pmod{n}}} K_{k+i}
\quad\text{and}\quad
K_{k+i-1} \subset K\left(J_{k+i-1 \pmod{n}},J_{k+i \pmod{n}}\right),
\]
and
\[
K_{k+j-1} \signedarrowequal{s_{k+i-1 \pmod{n}}} K_{k+j}
\quad\text{and}\quad
K_{k+j-1} \subset K\left(J_{k+i-1 \pmod{n}},J_{k+j \pmod{n}}\right).
\]
Thus,
$K_{k+i-1} \le K_{k+j-1}$ if and only if
$K_{k+i} \le K_{k+j}$ and $s_{k+i-1 \pmod{n}} = +.$
So, $K_{k+i-1} \le K_{k+j-1}$ if and only if
$J_{k+i \pmod{n}} < J_{k+j \pmod{n}}$ and $s_{k+i-1 \pmod{n}} = +.$
By iterating this argument $k-1$ times backwards we get that
$K_{i} \le K_{j}$ if and only if $J_{k+i \pmod{n}} < J_{k+j \pmod{n}}$
and
\[
\Sign\left(J_{i} \signedarrow{s_i}
           J_{i+1} \signedarrow{s_{i+1}} \cdots
           J_{k+i-1} \signedarrow{s_{k+i-1}}
           J_{k+i}\right) =
s_{i}s_{i+1} \cdots s_{k+i-1} = +
\]
(where every sub-index in the above formula is modulo $n$).
This ends the proof of the claim.

Observe that, since $K_n \subset J_0,$ from the construction
of the sets $K_n$ we get that
$K_{n-1} \subset K_{2n-1},K_{n-2} \subset K_{2n-2},\dots, K_0\subset K_n$
and
$K_{0} \signedarrowequal[F^n]{\Sign(\alpha)} K_{n}.$
Then, by Lemma~\ref{prop-band}(a.2,b) there exists a band
$C \subset K_0 \subset J_0$ such that
$
C \signedarrowequal[F^n]{\text{\scalebox{.7}{$\Sign(\alpha)$}}} C
$
and $F^i(C) \subset K_i \subset J_i$ for $i=0,1,\dots,n-1.$

Since $C$ is a periodic strip, $F^i(C)$ and $F^j(C)$ are either disjoint or equal.
Hence, by the claim,
$F^i(C) < F^j(C)$ if and only if $S^i(\alpha) < S^j(\alpha).$
Now, Lemma~\ref{simpleisdifferent} tells us that
$S^i(\alpha) \ne S^j(\alpha)$ whenever $i \ne j.$
Consequently, $F^i(C) \ne F^j(C)$ whenever $i \ne j$ and $C$ has period $n.$
This ends the proof of the lemma.
\end{proof}

\begin{remark}\label{4horseshoe}
From the construction in the above proof it follows that if
$F \in \cSO$ and
\[
\alpha = J_0 \signedarrow{s_{0}} J_1 \signedarrow{s_{1}} \cdots
J_{n-1} \signedarrow{s_{n-1}} J_{0}
\]
is a loop in the a signed Markov graph of $F$ by basic bands,
then there exist bands
$K_0 = K_0(\alpha) \subset J_0,\ K_1 \subset J_1,\ \ldots,\ K_{n-1} \subset J_{n-1}$
such that $K_{i} \signedarrowequal{s_i} K_{i+1}$ for $i=0,1,\dots, n-2$
and $K_{n-1} \signedarrowequal{s_{n-1}} J_0$.
In particular,
$
K_0 \signedarrowequal[F^n]{\text{\scalebox{.7}{$\Sign(\alpha)$}}} J_0.
$
Moreover, if $\beta$ is another loop such that $\alpha^\infty \ne \beta^\infty,$
then $K_0(\alpha)$ and $K_0(\beta)$ have pairwise disjoint interiors.
\end{remark}

\subsection{Proof of Theorem~\ref{teo-prin-A}}

We start this subsection with a lemma that studies the periodic
orbits of the uncoupled quasiperiodically forced skew-products on the
cylinder (in particular for the maps $F_\tau$).

\begin{lemma}\label{patternequiv}
Let $f \in \cI$ and let $F$ be a map from $\cSO$ such that
$F(\theta,x) = (R_{\omega}(\theta),f(x)).$
Then, the following statements hold.
\begin{enumerate}[(a)]
\item Assume that $P = \{p_{1},p_{2},\ldots,p_{n}\}$ is a periodic
orbit of $f$ with pattern $\tau.$
Then $\SI \times P$ is a periodic orbit of $F$ with pattern $\tau.$

\item If $B$ is a periodic orbit of strips of $F$ with pattern $\tau$ then
there exists a periodic orbit $P$ of $f$ with pattern $\tau$ such
that $\SI \times P$ is a periodic orbit of $F$ with pattern $\tau$ and
$\SI \times P \subset B.$
In particular, every cyclic permutation is a pattern of a function of
$F \in \cSO.$
\end{enumerate}
\end{lemma}

\begin{proof}
The first statement follows directly from the definition of a pattern.
Now we prove (b).
Let $B = \{B_{1},B_{2},\ldots,B_{n}\}$ be periodic orbit of strips of
$F$ with pattern $\tau$ (that is, $F(B_{i}) = B_{\tau(i)}$ for
$i=1,2,\dots,n$).
Since $F = (R_{\omega},f)$ it follows that $F^k = (R^k_{\omega},f^k)$
for every $k\in \N$ (so the iterates of $F$ are also uncoupled
quasiperiodically forced skew-products).
So, since $F^n(B_i) = B_i$ for every $i$, it follows that the
strips $B_i$ are horizontal. That is, for every $i$ there exists a
closed interval $J_i \subset \I$ such that $B_i = \SI \times J_i.$
Moreover, since the strips are pairwise disjoint, so are the
intervals $J_i$. Clearly, $f(J_{i}) = J_{\tau(i)}$ for every $i$ and,
hence, $f^{n}(J_{1}) = J_{1}.$
So, there exists a point $p_1 \in J_{1}$ such that
$f^{n}(p_1) = p_1$ and
$f^{k}(p_1) \in f^{k}(J_{1}) = J_{\tau^k(1)}$ for $k \ge 0.$
Since the intervals $J_i$ are pairwise disjoint,
the set $P = \{p_1, f(p_1),\dots, f^{n-1}(p_1)\}$ is a periodic orbit
of $f$ of period $n$ such that $\SI \times P \subset B.$
Moreover, if we set $f^{k}(p_1) = p_{\tau^k(1)}$ for
$k=1,2,\dots,n-1$, then $P$ has the spatial labelling and it follows
that the pattern of $P$ is $\tau.$
\end{proof}

\begin{proof}[Proof of Theorem~\ref{teo-prin-A}]
First we prove that $\tau \Longrightarrow_{\Omega} \nu$
implies $\tau \Longrightarrow_{\I} \nu.$
The assumption $\tau \Longrightarrow_{\Omega} \nu$
implies that every map $F \in \cSO$ that exhibits the strips
pattern~$\tau$ also exhibits the strips pattern~$\nu$.
In particular, the map $F_{\tau}$ has a periodic orbit of strips with
pattern~$\nu.$
By Lemma~\ref{patternequiv}, $f_{\tau}$ has a periodic orbit with
pattern~$\nu.$
Therefore, $\tau \Longrightarrow_{\I}\nu$ by the characterization of
the forcing relation in the interval (Theorem~\ref{carat-forc}).

Now we prove that $\tau \Longrightarrow_{\I} \nu$
implies $\tau \Longrightarrow_{\Omega} \nu.$
Clearly, we may assume that $\nu \ne \tau$.
We have to show that every $F \in \cSO$ that has a periodic orbit of strips
$B = \{B_0,B_1,\ldots,B_{n-1}\}$ with strips pattern $\tau$ also
has a periodic orbit of strips with strips pattern $\nu.$

We consider the map $f_{\tau} = f_P$ where $P$ is a periodic orbit
with pattern $\tau.$ By Theorem~\ref{carat-forc}, $f_{\tau}$ has
periodic orbit $Q = \{q_{0}, q_{1}, \dots, q_{n-1}\}$ with
pattern~$\nu.$ Since $Q$ has the spatial labelling, $q_0 = \min Q,$

Since $\nu \ne \tau$, by Lemma~\ref{lazo-unico}, there exists a simple
loop
\[
\alpha = I_{0} \signedarrow{s_0} I_{1}
 \signedarrow{s_1} \cdots \signedarrow{s_{n-2}} I_{n-1}
 \signedarrow{s_{n-1}}  I_0
\]
in the $P$-signed Markov graph of $f_{\tau}$ associated to $q_0.$
Moreover, by Definition~\ref{asoc-int},
\[
\begin{array}{lcl}
q_{0} &\sim&
  I_{0} \signedarrow{s_0} I_{1} \signedarrow{s_1} \cdots
        \signedarrow{s_{n-2}} I_{n-1} \signedarrow{s_{n-1}}  I_0 \\
f_{\tau}(q_{0}) &\sim&
  I_{1} \signedarrow{s_1} I_2 \signedarrow{s_2} \cdots
        \signedarrow{s_{n-1}} I_{0} \signedarrow{s_0} I_1\\
f_{\tau}^{2}(q_{0}) &\sim&
   I_{2} \signedarrow{s_2} I_3 \signedarrow{s_3} \cdots
        \signedarrow{s_{n-1}} I_{0} \signedarrow{s_0} I_1
        \signedarrow{s_1} I_2\\
\hspace{1em}\vdots && \hspace{8em}\vdots  \\
f_{\tau}^{n-1}(q_{0}) &\sim&
   I_{n-1} \signedarrow{s_{n-1}} I_0 \signedarrow{s_0} I_1
           \signedarrow{s_1} I_{2} \cdots
           \signedarrow{s_{n-2}} I_{n-1},
\end{array}
\]
where the symbol $\sim$ means ``associated with''.
By Remark~\ref{graphunique-strips} (see also
Remark~\ref{graphunique-int}), the above loop $\alpha$ also
exists in the $B$-signed Markov graph of $F$ by replacing the
basic intervals~$I_{i} = [q_i, q_{i+1}]$
by the basic bands~$I_{_{B_{i}B_{i+1}}}:$
\[
\alpha = I_{_{B_{0}B_{1}}} \signedarrow{s_0}
         I_{_{B_{1}B_{2}}} \signedarrow{s_1} \cdots
\signedarrow{s_{n-2}}
         I_{_{B_{n-2}B_{n-1}}} \signedarrow{s_{n-1}}
I_{_{B_{0}B_{1}}}.
\]
By Lemma~\ref{lazoasociadobandas} and Definition~\ref{asoc-band},
$F$ has a periodic band $Q_0$ associated to $\alpha$
(and hence of period $n$), and
\begin{align*}
Q_{0} &\sim&&\hspace*{-0.7em}
         I_{_{B_{0}B_{1}}} \signedarrow{s_0}
         I_{_{B_{1}B_{2}}} \signedarrow{s_1} \cdots
\signedarrow{s_{n-2}}
         I_{_{B_{n-2}B_{n-1}}} \signedarrow{s_{n-1}}
I_{_{B_{0}B_{1}}}\\
F(Q_{0}) &\sim&&\hspace*{-0.7em}
         I_{_{B_{1}B_{2}}} \signedarrow{s_1}
         I_{_{B_{2}B_{3}}} \signedarrow{s_2} \cdots
\signedarrow{s_{n-2}}
         I_{_{B_{n-2}B_{n-1}}} \signedarrow{s_{n-1}}
         I_{_{B_{0}B_{1}}}\signedarrow{s_0} I_{_{B_{1}B_{2}}}  \\
F^{2}(Q_{0}) &\sim&&\hspace*{-0.7em}
         I_{_{B_{2}B_{3}}} \signedarrow{s_2} \cdots
\signedarrow{s_{n-2}}
         I_{_{B_{n-2}B_{n-1}}} \signedarrow{s_{n-1}}
         I_{_{B_{0}B_{1}}} \signedarrow{s_0}
         I_{_{B_{1}B_{2}}} \signedarrow{s_1} I_{_{B_{2}B_{3}}} \\
\hspace{1em}\vdots &&& \hspace{8em}\vdots  \\
F^{n-1}(Q_{0}) &\sim&&\hspace*{-0.7em}
         I_{_{B_{n-2}B_{n-1}}} \signedarrow{s_{n-1}}
         I_{_{B_{0}B_{1}}} \signedarrow{s_0}
         I_{_{B_{1}B_{2}}} \signedarrow{s_1} \cdots
\signedarrow{s_{n-2}}
         I_{_{B_{n-2}B_{n-1}}}.
\end{align*}

By Lemmas~\ref{lazosintervalasociados} and \ref{simpleisdifferent},
the order of the points $f^{i}_{\tau}(q_{0})$ induces an order on
the shifts of the loop $S^{i}(\alpha),$ with the usual lexicographical
ordering and, by Lemma~\ref{lazoasociadobandas}, the order of the shifts
$S^{i}(\alpha)$ induces the same order on the bands $F^{i}(Q_{0}).$
Thus, for every $i,j \in \{0,1,\dots,n-1\},$ $i \ne j,$
$F^{i}(Q_{0}) < F^{j}(Q_{0})$ if and only if
$f_\tau^{i}(q_{0}) < f_\tau^{j}(q_{0}).$
So, $\{Q_{0},F(Q_0), F^2(Q_0),\dots,F^{n-1}(Q_0)\}$
and $\{q_{0}, q_{1}, \dots, q_{n-1}\}$ have the same pattern $\nu.$
This concludes the proof.
\end{proof}

\section{Proof of Theorems~\ref{teo-prin-B} and
\ref{teo-prin-C}}\label{ProofOfCandD}

We start by proving Theorem~\ref{teo-prin-B}.

The next technical lemma is inspired in \cite[Lemma~4.3.1]{ALM}.

\begin{lemma}\label{hor1}
Let $F \in \cSO$ and let $(J,\mathcal{D})$ be an $s$-horseshoe of
$F.$
Then, there exists $\mathcal{D}_n$,
a set of $s^{n}$ pairwise weakly ordered bands contained in $J,$
each of them with non-empty interior, such that
$(J,\mathcal{D}_n)$ is a $s^{n}$-horseshoe for $F^{n}.$
\end{lemma}

\begin{proof}
We use induction. For $n = 1$ there is nothing to prove.

Suppose that the induction hypothesis holds for some $n$ and
let $D \in \mathcal{D}$ and $C \in \mathcal{D}_{n}.$
Since $C \subset J$ has non-empty interior and $D \signedarrow{\pm} J,$
by Lemma~\ref{prop-band}(a.2,3),
there exists a band $B(D,C) \subset D$ with non-empty interior
such that $B(D,C) \signedarrowequal{\pm} C.$
Moreover, given $C' \in \mathcal{D}_{n}$ with $C' \ne C,$
$B(D,C)$ and $B(D,C')$ can be chosen to be weakly ordered
because $C$ and $C'$ are weakly ordered by assumption.
Since, $C \in \mathcal{D}_{n},$ $B(D,C) \signedarrow[F^{n+1}]{\pm} J.$
Thus, the family
\[
\mathcal{D}_{n+1} =
\set{B(D,C)}{D \in \mathcal{D} \text{ and } C \in \mathcal{D}_n}
\]
consists of $s^{n+1}$ pairwise weakly ordered bands contained in $J,$
each of them with non-empty interior,
such that $B(D,C)$ $F^{n+1}$-covers $J.$
Consequently, $(J,\mathcal{D}_{n+1})$ is an $s^{n+1}$-horseshoe for
$F^{n+1}.$
\end{proof}

\begin{proof}[Proof of Theorem~\ref{teo-prin-B}]
Fix $n > 0.$
By Lemma~\ref{hor1}, $F^{n}$ has a $s^{n}$-horseshoe
$(J,\mathcal{D}).$
Remove the smallest and the biggest band of
$\mathcal{D}$ and call $K$ the smallest band that contains the
remaining elements of $\mathcal{D}.$
Clearly, $K$ is contained in the interior of $J.$
Thus, by Lemma~\ref{prop-band}(a.2,3), each element $D$ of
$\mathcal{D}$ contains in its interior a band $A(D)$
such that $A(D) \signedarrowequal[F^{n}]{\pm} K.$
Then there exists an open cover $\cB$ of the strip $J$
(formed by open sets $B$ such $B^{\theta}$
is an open interval for every $\theta \in \SI$),
such that for each $D \in \mathcal{D}\evalat{K},$
the band $A(D)$ intersects only one element $B(D)$ of $\cB$
(then it has to be contained in it) and if $D, D' \in
\mathcal{D}\evalat{K}$ with $D \ne D'$ then $B(D)\ne B(D').$
For $D_0,D_1,\ldots,D_{k-1} \in \mathcal{D}\evalat{K}$
the set
$\cap_{i = 0}^{k-1} F^{-n}(A(D_i))$ is non-empty and intersects only
one element of $\cB_{F^{n}}^{k},$
namely $\cap_{i = 0}^{k-1} F^{-n}(B(D_i)).$
Therefore the sets $\cap_{i = 0}^{k-1} F^{-n}(A(D_i))$ are different
for different sequences $(D_0,D_1,\ldots,D_{k-1}),$
and thus
\[
 \mathcal{N}(\cB_{F^{n}}^{k}) \ge (\Card \mathcal{D}-2)^{k},
\]
where $\mathcal{N}(\cB_{F^{n}}^{k})$ is defined as in
\cite[Section~4.1]{ALM}.
Hence,
\[
h(F) = \frac{1}{n} h(F^{n}) \ge
       \frac{1}{n} h(F^{n}, \cB) \ge
       \frac{1}{n} \log(\Card(\mathcal{D}) - 2) =
       \frac{1}{n} \log(s^{n}-2).
\]
Since $n$ is arbitrary, we obtain $h(F) \ge \log(s).$
\end{proof}

Now we aim at proving Theorem~\ref{teo-prin-C}.
To this end we have to introduce some more notation and preliminary
results concerning the \emph{topological entropy}.

Given a map $f \in \cSO$, $h(F\evalat{I_{\theta}})$ is defined for
every $I_{\theta} := \{\theta\} \times \I$
(despite of the fact that it is not $F$-invariant)
by using the Bowen definition of the topological entropy
(c.f. \cite{Bowen, BowenErr}).
Moreover, the Bowen Formula gives
\[
\max \{h(R), h_{\operatorname{fib}}(F) \} \le h(F) \le
    h(R) + h_{\operatorname{fib}}(F)
\]
where
\[
h_{\operatorname{fib}}(F) =
  \sup\nolimits_{\theta \in \SI} h(F\evalat{I_{\theta}}).
\]
Since $h(R) = 0$, it follows that
$h(F) = h_{\operatorname{fib}}(F).$

In the particular case of the uncoupled maps
$F_{\tau} = (R, f_\tau)$ we easily get the following result:

\begin{lemma}\label{Ft=ft}
Let $\tau$ be a pattern (both in $\I$ and $\Omega$). Then
$h(F_{\tau}\evalat{I_{\theta}}) = h (f_{\tau})$
for every $\theta \in \SI.$
Consequently,
\[
h(F_{\tau}) = h_{\operatorname{fib}}(F_{\tau}) = h (f_{\tau}).
\]
\end{lemma}

Given a signed Markov graph $G$ with vertices $I_1, I_2,\dots, I_n$
we associate to it a $n \times n$ \emph{transition matrix}
$T_G = (t_{ij})$ by setting $t_{ij} = 1$ if and only if
there is a signed arrow from the vertex $I_i$ to the
vertex $I_j$ in $G.$ Otherwise, $t_{ij}$ is set to 0.

The spectral radius of a matrix $T,$ denoted by $\rho(T),$ is equal
to the maximum of the absolute values of the eigenvalues of $T.$

\begin{lemma}\label{ent-mat}
Let $P$ be a periodic orbit of strips of $F \in \cSO$ and let
$T$ be the transition matrix of the $P$-signed Markov graph of $F.$
Then
\[
 h(F) \ge \max\{0, \log(\rho(T))\}.
\]
\end{lemma}

\begin{proof}
If $\rho(T) \le 1$ then there is nothing to prove.
So, we assume that $\rho(T) > 1.$
Let $J$ be the $i$-th $P$-basic band and let $s$ be the $i$-th
entry of the diagonal of $T^{n}.$
By \cite[Lemma 4.4.1]{ALM} there are $s$ loops of length $n$ in the
$P$-signed Markov graph of $F$ beginning and ending at $J.$
Hence, if $s \ge 2,$ $F^{n}$ has an $s$-horseshoe $(J,\mathcal{D})$
by Remark~\ref{4horseshoe}. By Theorem~\ref{teo-prin-B},
$h(F) = \tfrac{1}{n} h(F^{n}) \ge \tfrac{1}{n} \log(s).$

If there are $k$ basic bands, the trace of $T^{n}$ is not larger
than $k$ times the maximal entry  on the diagonal of $T^{n}.$
Hence,
$
h(F) \ge \tfrac{1}{n} \log\left(\tfrac{1}{k} \tr(T^{n}) \right).
$
Therefore, by \cite[Lemma~4.4.2]{ALM},
\[
 h(F) \ge \limsup_{n\to\infty} \frac{1}{n}
          \log\left(\frac{1}{k}\tr(T^{n})\right) =
       \limsup_{n\to\infty} \frac{1}{n} \log(\tr(T^{n}) =
       \log(\rho(T)).
\]
\end{proof}

\begin{proof}[Proof of Theorem~\ref{teo-prin-C}]
Let $P$ be a periodic orbit of strips with pattern $\tau$ and let $T$
be the transition matrix of the $P$-signed Markov graph of $F.$
Let $f_\tau = f_Q$ be a $Q$-linear map in $\cI,$
where $Q$ is a periodic orbit of $f_Q$ with pattern $\tau.$
In view of Remark~\ref{graphunique-strips} (see also
Remark~\ref{graphunique-int}),
$T$ is also the transition matrix of the $Q$-signed Markov graph of
$f_\tau.$
Consequently, by \cite[Theorem~4.4.5]{ALM},
$
h(f_{\tau}) = \max \{0, \log\left(\rho(T)\right)\}.
$
By Lemmas~\ref{ent-mat} and \ref{Ft=ft},
\[
h(F) \ge \max\{0, \log(\rho(T))\} = h(f_{\tau}) = h(F_{\tau}).
\]
\end{proof}

\bibliographystyle{plain}
\bibliography{ForcingAnnulus}
\end{document}